\documentclass[a4paper,11pt]{article}%
\usepackage{amsmath,amssymb,amsfonts,amsthm}
\usepackage{setspace}
\usepackage{fullpage}
\usepackage{hyperref}
\usepackage{amscd}
\usepackage{graphpap}
\usepackage[pdftex]{color,graphicx}
\usepackage{color,graphicx}
\usepackage{empheq}
\usepackage{ragged2e}
\usepackage{comment}
\usepackage{array,epsfig}
\usepackage{amsxtra}
\usepackage{mathrsfs}
\usepackage{xcolor}
\usepackage{natbib}
\usepackage{makeidx}
\usepackage[justification=centering]{caption}
\usepackage{enumerate}
\usepackage{nccmath}
\usepackage{empheq}
\usepackage{leading} 
\usepackage{subfigure} 
\usepackage{graphicx}
\usepackage{caption,multicol,booktabs,array}
\usepackage[singlelinecheck=false]{caption}
\usepackage[labelfont=bf,labelsep=quad,justification=justified]{caption}
\usepackage{array}
\usepackage{float}
\usepackage{caption}
\usepackage{booktabs}
\usepackage{arydshln}
\usepackage{hyperref}
\usepackage{indentfirst}
\usepackage{mathpazo}
\newtheorem{theorem}{Theorem}[section]
\newtheorem{definition}{Definition}[section]

\newtheorem{proposition}[theorem]{Proposition}
\newtheorem{corollary}[theorem]{Corollary}
\newtheorem{remark}{Remark}

\usepackage{geometry}
 \hypersetup{
  colorlinks=true,
  linkcolor=red,
  citecolor=blue,
  filecolor=blue,
  urlcolor=teal,
}

\title{\bf Fractional Poisson random sum and its associated normal variance mixture}
\author{Gabriela Oliveira$^\S$\footnote{Email: \href{mailto:gabriela.oliveira.mat@gmail.com}{gabriela.oliveira.mat@gmail.com}},\,\,  Wagner Barreto-Souza$^\P$\footnote{Corresponding author. Email: \href{mailto:wagner.barretosouza@kaust.edu.sa}{wagner.barretosouza@kaust.edu.sa}}\,\, and Roger W.C. Silva$^\S$\footnote{Email: \href{mailto:rogerwcs@est.ufmg.br}{rogerwcs@est.ufmg.br}} \\\\
\small $^\S$\it Departamento de Estat\' \i stica, Universidade Federal de Minas Gerais, Belo Horizonte, Brazil\\
\small $^\P$\it Statistics Program, King Abdullah University of Science and Technology, Thuwal, Saudi Arabia}
\date{}

\begin{document}
\sloppy

\maketitle

\begin{abstract}
In this work, we study the partial sums of independent and identically distributed random variables with the number of terms following a fractional Poisson (FP) distribution. The FP sum contains the Poisson and geometric summations as particular cases. We show that the weak limit of the FP summation, when properly normalized, is a mixture between the normal and Mittag-Leffler distributions, which we call by Normal-Mittag-Leffler (NML) law. A parameter estimation procedure for the NML distribution is developed and the associated asymptotic distribution is derived. Simulations are performed to check the performance of the proposed estimators under finite samples. An empirical illustration on the daily log-returns of the Brazilian stock exchange index (\texttt{IBOVESPA}) shows that the NML distribution captures better the tails than some of its competitors. Related problems such as a mixed Poisson representation for the FP law and the weak convergence for the Conway-Maxwell-Poisson random sum are also addressed.  \\


\end{abstract}

\section{Introduction}\label{intro}

The sum of a random number of random variables constitutes a very important tool in statistical analysis. Besides being interesting from a probabilistic point of view, it also appears in a wide range of applications involving processes that evolve with time such as insurance models, queuing theory, finance, reliability theory, biology, among others. See \cite{gk} and \cite{k} for a comprehensive list of such applications.  

Let $N$ be a non-negative integer-valued random variable and $\{W_j\}_{j\in\mathbb{N}}$ a sequence of independent and identically distributed (${\it i.i.d.}$) random variables, independent of $N$. Then, a random summation is given by 
\begin{equation}\label{random_sum}
Y=\sum_{j=1}^{N}W_j,
\end{equation} 
where $Y\equiv 0$ when $N=0$. An important example of random sum is the compound Poisson distribution, which is obtained by taking $N$ following a Poisson distribution with mean $\zeta>0$ in (\ref{random_sum}). In this case, it can be shown that (\ref{random_sum}) properly normalized  weakly converges to the standard normal distribution as $\zeta \rightarrow \infty$. 

The geometric random sum, obtained by assuming $N$ geometrically distributed (with mean $1/p$ for $p\in(0,1)$) in (\ref{random_sum}) was studied for instance by \cite{ren1956}, \cite{kozrac1994}, \cite{kozrac1999}, \cite{kotetal2001}, \cite{k} and \cite{jorkok2011}. By taking $p\rightarrow0^+$, (\ref{random_sum}) properly normalized converges in distribution to the Laplace law. Generalizations of the geometric random summation were proposed and studied by assuming that $N$ belongs to the class of mixed Poisson distributions \citep{karxek2005}, where the weak limits are normal variance mixture distributed \citep{baretal1982}. Contributions on this direction are due to \cite{korshe2012}, \cite{korzei2016}, \cite{st}, \cite{she2018}, and \cite{olietal2020}, just to name a few.

Our chief goal in this paper is to study the random sum in (\ref{random_sum}) when $N$ follows a fractional Poisson (FP) distribution. In this case, we refer to (\ref{random_sum}) as {\it fractional Poisson summation}. The FP distribution is a generalization of the Poisson law introduced by \cite{repsai2000}. This count model is obtained as the marginal distribution of a renewal process with Mittag-Leffler \citep{p} waiting times, the so-called fractional Poisson process. \cite{l} showed that the fractional Poisson process captures the long-memory effect which results in non-exponential waiting time distribution empirically observed in complex systems. Some applications of the fractional Poisson process, such as quantum physics and combinatorial number theory, are presented in \cite{l2}. Additional works exploring the FP process are due to \cite{maietal2004}, \cite{cahetal2010}, \cite{meeetal2011}, \cite{leoetal2017}, \cite{mahvel2019}, among others.

Random sums involving the fractional Poisson distribution/process have already been studied in the literature. We refer the reader to the works by \cite{sca2011}, \cite{begmac2013}, \cite{begmac2014}, and \cite{biasau2014}, where stochastic properties and representations are obtained. Our paper goes in a different direction. We are interested in obtaining the weak limit of the FP summation under a proper normalization. We will show that the resulting limit law is a new normal variance mixture involving the Mittag-Leffler distribution, which we call by Normal-Mittag-Leffler (in short NML) model. Moreover, we explore the statistical properties of this new law and show through an empirical illustration that it can be a good alternative for modeling financial data in comparison with existing models.

Other contributions of our paper are the discussion of two related problems: (i) as a result of the weak convergence of the FP sum, we show that the fractional Poisson distribution admits a mixed Poisson representation, which is a new finding; (ii) we show that the weak limit of a Conway-Maxwell-Poisson \citep{cm,smkbb} sum (under proper normalization) is normally distributed. In particular, this last point illustrates that not all extensions of the Poisson distribution considered in (\ref{random_sum}) results in a non-normal asymptotic distribution.

This paper is organized as follows. In Section \ref{sec:FPRS}, we establish the weak limit of a properly normalized fractional Poisson random sum and show that it is a mixture between the normal and the Mittag-Leffler distributions. Statistical properties of the NML law are explored. Section \ref{sec:estimation} is devoted to the parameter estimation of the NML model and asymptotic distribution of the estimators as well. Numerical experiments aiming at the investigation of the proposed estimation procedure under finite-samples are also presented and a real data application on the daily log-returns of Brazilian stock exchange \texttt{IBOVESPA} are provided in Section \ref{sec:sim_app}. Two related problems are addressed in Section \ref{sec:relatedproblems}. 

\section{Fractional Poisson random sum: Convergence and properties}\label{sec:FPRS}

In this section, we obtain the weak limit of a fractional Poisson (FP) random summation when properly normalized and study the properties of its limiting distribution. We begin by presenting the FP distribution/process and some of its basic properties.

\subsection{Fractional Poisson distribution}

The fractional Poisson process was introduced in \cite{repsai2000} as a non-markovian renewal process. More specifically, let $\{\Delta_{T_k}\}_{k\geq 1}$ be a sequence of $i.i.d.$ waiting times with cumulative distribution function
\begin{eqnarray}\label{cdf_type_I_ML}
P(\Delta_{T_k}\leq t)=1-\mathcal{E}_{\kappa}(-\nu t^{\kappa}),\quad t>0,
\end{eqnarray}
for $\nu>0$ and $0<\kappa\leq 1$, where $\mathcal{E}_{\kappa}(\cdot)$ is the Mittag-Leffler function defined by
\begin{equation}\label{ml_function}
\mathcal{E}_{\kappa}(z) = \sum\limits_{m = 0}^{\infty} \dfrac{z^m}{\Gamma(\kappa m + 1)},\quad z \in \mathbb{R},
\end{equation}
where $\Gamma(\cdot)$ is the gamma function. For more details on the Mittag-Leffler function, see \cite{gkmr}. A random variable with distribution function given by (\ref{cdf_type_I_ML}) is said to follow a type 1 Mittag-Leffler distribution; see \cite{hui2016}. For $\kappa=1$, the exponential distribution is obtained as a particular case. Write $T_n=\Delta_{T_1}+\Delta_{T_2}+\dots+\Delta_{T_n}$ for the time of the $n$-th jump and let $$N_{\nu,\kappa}(t)=\sup\{n\geq 0: T_n\leq t\}.$$ 

Then, $\{N_{\nu,\kappa}(t)\}_{t\geq 0}$ is a renewal process with type 1 Mittag-Leffler waiting times, which is called by fractional Poisson process, with parameters $\nu$ and $\kappa$. Let $P_{\nu,\kappa}(n,t)$ be the probability of $n$ arrivals up to time $t$. \cite{l} has shown that
$$P_{\nu,\kappa}(n,t) = \frac{(\nu t^{\kappa})^n}{n!}\sum\limits_{i = 0}^{\infty}\frac{(i + n)!}{i!}\frac{(-\nu t^{\kappa})^i}{\Gamma(\kappa(i + n) + 1)}, \; n=0,1,2\dots.$$

The mean and variance of $N_{\nu,\kappa}(t)$ are respectively given by
\begin{eqnarray}\label{moments}
E(N_{\nu,\kappa}(t)) = \frac{\nu t^{\kappa}}{\Gamma(\kappa + 1)}\quad\mbox{and}\quad
\mbox{Var}(N_{\nu,\kappa}(t)) = E(N_{\nu,\kappa}(t)) + E^2(N_{\nu,\kappa}(t))\left\{\frac{\kappa B(\kappa, 1/2)}{2^{2\kappa - 1}} - 1\right\},
\end{eqnarray}
where $B(a, b) = \dfrac{\Gamma(a)\Gamma(b)}{\Gamma(a+b)}$ is the beta function with $a > 0$ and $b > 0$. For more details on the FP process, we refer to \cite{l}.

We use $\mbox{FPP}(\nu, \kappa)$ to denote a fractional Poisson process with parameters $\nu$ and $\kappa$. Note that the parameter $\nu$ is also related to the asymptotic distribution of the inter-arrival times of the fractional Poisson process, which is a power-law with parameter $\nu$. 
The FPP has the Poisson process as a particular case when $\kappa=1$.

For fixed $t>0$, we have that the probability generating function of $N_{\nu,\kappa}(t)$ is given by
\begin{equation}\label{pgf_fp}
    G_{N_{\nu,\kappa}(t)}(s) = E(s^{N_{\nu,\kappa}(t)}) = \mathcal{E}_{\kappa}\big(\nu t^{\kappa}(s - 1)\big),\quad s\in\mathbb R.
\end{equation}

In what follows, we consider a fractional Poisson random variable $N\equiv N_{\nu,\kappa}(1)$ (FP process at time $t=1$) and denote $N\sim\mbox{FP}(\nu,\kappa)$. The first two cumulants and probability generating function of the FP distribution are obtained respectively from (\ref{moments}) and (\ref{pgf_fp}) with $t=1$.

\subsection{Weak limit and properties}

We now obtain the characteristic function of the limit distribution of a fractional Poisson random sum under a proper normalization. We show that the resulting limit distribution is a new type of normal variance mixture. In the sequence, a random variable $U$ with a type 2 Mittag-Leffler (ML) distribution depending on the parameter $\kappa\in(0,1)$, denoted by $U\sim ML(\kappa)$, has density function 
\begin{equation}\label{densidade mittagleffler}
f_{\kappa}(u) = \dfrac{1}{\pi\kappa}\sum\limits_{j = 1}^{\infty}\frac{(-1)^{j-1}}{j!}\sin(\pi\kappa j)\Gamma(\kappa j + 1)u^{j-1},\quad u>0,
\end{equation}
and moment generation function given by $E(\exp(sU))=\mathcal{E}_{\kappa}(s)$, for $s\in\mathbb R$. For more details on the type 2 Mittag-Leffler law, see \cite{hui2016}.  From now on, we refer to this model as simply Mittag-Leffler; not to be confused with the type 1 version. We have that $\lim_{\kappa\rightarrow 1^-}E(\exp(sU))=\mathcal{E}_{1}(s)=\exp(s)$, for $s\in\mathbb R$. Therefore, the (type 2) ML distribution converges in distribution to a degenerate random variable at 1 as $\kappa\rightarrow1^-$.

The notation $V_1\overset{d}{=}V_2$ means the random variables $V_1$ and $V_2$ follow the same distribution. 

\begin{proposition}\label{convergencia distribuicao padrao ML}
Let $N\sim\mbox{FP}(\nu, \kappa)$ and $\{W_n\}_{n\in\mathbb{N}}$ a sequence of $i.i.d.$ random variables independent of  $N$ with $E(W_1) = 0$ and $\mbox{Var}(W_1) = 1$. Then, \\

\noindent (i) it holds that
\begin{equation}\label{limit}\widetilde{S}_{\nu} \equiv %\frac{S_N - E(S_N)}{\sqrt{\nu}} = \frac{S_N}{\sqrt{\nu}}  
\frac{1}{\sqrt{\nu}}\sum\limits_{i = 1}^{N} W_i \xrightarrow[\nu \rightarrow \infty]{d} Y,
\end{equation}
where $Y$ has characteristic function given by
\begin{eqnarray}\label{fch}
\varphi_Y(s) = \mathcal{E}_{\kappa}\left(-\frac{s^2}{2}\right), \quad s \in \mathbb{R}, \; \; \kappa \in (0, 1];
\end{eqnarray}

\noindent (ii) the limit random variable $Y$ in \eqref{limit} admits the stochastic representation
\begin{equation}\label{representacao funcao caracteristica}
    Y \overset{d}{=} \sqrt{U}Z,
\end{equation}
where $Z \sim N(0, 1)$ and $U \sim ML(\kappa)$ are independent.

\end{proposition}

\begin{proof}
(i) We will show the convergence of the characteristic functions and then the result will follow from the L\'evy's Continuity Theorem.

We first need to justify the given normalization of $S_\nu\equiv \sum\limits_{i = 1}^{N} W_i $. We have that $E(S_\nu)=E(N)E(W_1)=0$ and $\mbox{Var}(S_\nu)=E(\mbox{Var}(S_\nu|N))+\mbox{Var}(E(S_\nu|N))=E(\mbox{Var}(S_\nu|N))=E(N)=\dfrac{\nu}{\Gamma(\kappa+1)}$. Therefore, $\mbox{Var}(S_\nu)=\mathcal O(\nu^{1/2})$ and then $\widetilde{S}_{\nu}=\dfrac{S_\nu}{\sqrt{\nu}}$ will have a non-trivial limit in distribution. Using properties of conditional expectation, it holds that
\begin{eqnarray*}
\varphi_{\widetilde{S}_{\nu}}(s) &=& E\left\{E\left(\exp\left\{i\,s\,\nu^{-1/2}\sum\limits_{i = 1}^{N}W_i\right\} \Bigg| N\right)\right\} 
 = G_{N}\left(\varphi_{W_1}\left( \frac{s}{\sqrt{\nu}}\right)\right) 
 = \mathcal{E}_{\kappa}\left(\nu\left\{\varphi_{W_1}\left(\frac{s}{\sqrt{\nu}}\right) - 1\right\}\right),
\end{eqnarray*}
where $G_{N}(\cdot)$ is the probability generating function of $N$ given in (\ref{pgf_fp}) with $t=1$ and $\varphi_{W_1}(\cdot)$ is the characteristic function of $W_1$. 
Since $\mathcal{E}_{\kappa}(\cdot)$ is continuous, it follows that
\begin{equation}\label{lim_aux}
   \lim\limits_{\nu \to \infty} \varphi_{\widetilde{S}_{\nu}}(s) = \mathcal{E}_{\kappa}\left( \lim\limits_{\nu \to \infty}\nu\left\{\varphi_{W_1}\left(\frac{s}{\sqrt{\nu}}\right) - 1\right\}\right). 
\end{equation}

The limit above has an indeterminate form of the type ``$\infty\times 0$". We apply L'H\^opital's rule in (\ref{lim_aux}) to obtain that
\begin{eqnarray}\label{lim_aux2}
\lim\limits_{\nu \to \infty}\dfrac{\varphi_{W_1}\left(\frac{s}{\sqrt{\nu}}\right) - 1}{1/\nu}=\lim\limits_{\nu \to \infty}\dfrac{(-s/2)\nu^{-3/2}\varphi'_{W_1}\left(s/\sqrt{\nu}\right)}{-\nu^{-2}}=\lim\limits_{\nu \to \infty}\dfrac{(s/2)\varphi'_{W_1}\left(s/\sqrt{\nu}\right)}{\nu^{-1/2}},
\end{eqnarray}
where $\varphi'_{W_1}(x)=d\varphi_{W_1}(x)/dx$. Note that $\varphi'_{W_1}\left(0\right)=iE(W_1) = 0$ and, therefore, (\ref{lim_aux2}) has again the indeterminate form ``$0/ 0$".
A second application of L'H\^opital's rule gives us that 
\begin{eqnarray*}
\frac{s}{2}\lim\limits_{\nu \to \infty}\dfrac{\varphi'_{W_1}\left(s/\sqrt{\nu}\right)}{\nu^{-1/2}}&=&\frac{s}{2}\lim\limits_{\nu \to \infty}\dfrac{(-s/2)\nu^{-3/2}\varphi''_{W_1}\left(s/\sqrt{\nu}\right)}{-1/2\nu^{-3/2}}\\
&=&\frac{s^2}{2}\lim\limits_{\nu \to \infty}\varphi''_{W_1}\left(s/\sqrt{\nu}\right)\\
&=&\frac{s^2}{2}\varphi''_{W_1}\left(0\right) = \frac{s^2}{2}i^2E(W_1^2)=-\frac{s^2}{2},
\end{eqnarray*}
where $\varphi''_{W_1}(x)=d^2\varphi_{W_1}(x)/dx^2$. Hence,
\begin{equation*}
   \lim\limits_{\nu \to \infty} \varphi_{\widetilde{S}_{\nu}}(s) = \mathcal{E}_{\kappa}\left(- \frac{s^2}{2}\right), \; \forall s \in \mathbb{R}.  
\end{equation*}

\noindent (ii) In \cite{gkmr} (see  Section 3.7) it is shown that the Mittag-Leffler function with negative argument can be written as  
\begin{equation}\label{representacao M-L}
    \mathcal{E}_{\kappa}\left(-\frac{s^2}{2}\right) = \int\limits_{0}^{\infty}e^{-\frac{s^2}{2}u}f_{\kappa}(u)du,
\end{equation}
where $f_{\kappa}(\cdot)$ is the density function of a (type 2) Mittag-Leffler distribution given in (\ref{densidade mittagleffler}); see \cite{bm}. This indicates that the limiting distribution is a normal variance mixture involving a ML distribution. To prove this claim, let $U\sim\mbox{ML}(\kappa)$ independent of $Z\sim N(0,1)$. The characteristic function of $\sqrt{U}Z$ is
\begin{eqnarray}
\nonumber \varphi_{\sqrt{U}Z}(s) = E(e^{is\sqrt{U}Z}) &=& E\left(E\left(e^{is\sqrt{U}Z}|U\right)\right)=E\left(\varphi_Z(s\sqrt{U})\right)\\
\nonumber &=& E\left(e^{-\frac{s^2}{2}U}\right)=\int\limits_{0}^{\infty}e^{-\frac{s^2}{2}u}f_{\kappa}(u)du,\quad s\in\mathbb R,
\nonumber 
\end{eqnarray}
which coincides with the characteristic function of our limiting distribution in (\ref{representacao M-L}). Therefore, we have proven that $Y \overset{d}{=} \sqrt{U}Z$.
\end{proof}

\begin{definition}
A random variable $Y$ is said to follow a standard Normal-Mittag-Leffler distribution if satisfies the stochastic representation (\ref{representacao funcao caracteristica}). We denote by $Y\sim\mbox{NML}(\kappa)$.
\end{definition}

\begin{remark}
To the best of our knowledge, the NML distribution defined above as a normal variance mixture is a new law in the literature. Moreover, the stochastic representation in (\ref{representacao funcao caracteristica}) will be particularly useful in Section \ref{sec:simulation} when performing Monte Carlo simulation, where an NML random variable generator is required.
\end{remark}

\begin{remark}
From (\ref{fch}), it can be shown that our NML law weakly converges to the standard Laplace and normal distributions when $\kappa\rightarrow0^+$ and $\kappa\rightarrow1^-$, respectively. Therefore, our model makes a continuous bridge between these two well-known distributions. Similarly, we have that the FP sum contains the geometric and Poisson summations as $\kappa\rightarrow0^+$ and $\kappa\rightarrow1^-$, respectively.
\end{remark}

\begin{remark}
Our NML law is different from the Mittag-Leffler-Gaussian (MLG) distribution proposed by \cite{aa}, despite similar names. The MLG model was introduced by replacing the exponential function with the Mittag-Leffler function in the normal density function. In our case, the NML distribution naturally arises as a weak limit of a fractional Poisson random sum.
\end{remark}

\begin{remark} In \cite{hui2016}, the author observes that a random variable $U \sim ML(\kappa)$ can be represented by
\begin{eqnarray*} 
U \overset{d}{=} \dfrac{1}{Q^\kappa},
\end{eqnarray*}
where $Q$ is one-sided stable distributed with tail exponent $\kappa\in(0,1)$, with density function 
\begin{equation*}
f_Q(q)  = \dfrac{1}{\pi}\sum\limits_{j = 1}^{\infty}\frac{(-1)^{j-1}}{j!}\sin(\pi\kappa j)\Gamma(\kappa j + 1)q^{-(\kappa j+1)},\quad q>0.
\end{equation*}
Therefore, $Y$ can be represented as
\begin{eqnarray}\label{mixture NML}
Y \overset{d}{=} \dfrac{Z}{Q^{\kappa/2}},
\end{eqnarray}
where $Z \sim N(0, 1)$ and $Q$ are independent. The stochastic representation (\ref{mixture NML}) is an alternative to (\ref{representacao funcao caracteristica}).
\end{remark}

Let $Y\sim\mbox{NML}(\kappa)$. The moments of $Y$ are easily obtained through the characteristic function (\ref{fch}) and the sum representation of the Mittag-Leffler function in (\ref{ml_function}). Its moments of odd order are null, that is $E(Y^{2j+1})=0$ for all $j\geq 0$. The even moments are given by $E(Y^{2j}) = \dfrac{(2j)!}{2^j\Gamma(j\kappa + 1 )}$, for  $j \in \mathbb{N}$. In particular, the first four cumulants of the NML law are
\begin{equation*}
E(Y) = 0,\quad \mbox{Var}(Y) = \dfrac{1}{\Gamma(\kappa + 1)},\quad 
\gamma_1 = 0,\quad \mbox{and}\quad
\gamma_2 = \dfrac{6\Gamma^2(\kappa + 1)}{\Gamma(2\kappa + 1)} - 3,
\end{equation*}
where $\gamma_1$ and and $\gamma_2$ are the asymmetry coefficient and excess kurtosis, respectively.

The following limits indicate the asymptotic behavior of the variance and excess kurtosis of $Y$ as function of $\kappa$ when approaching the boundaries of the parameter space:
$$\left\{\begin{array}{rc}
\lim\limits_{\kappa \to 0} \mbox{Var}(Y) = \lim\limits_{\kappa \to 0}  \dfrac{1}{\Gamma(\kappa + 1)} = 1,\\ \noalign{\smallskip}
\lim\limits_{\kappa \to 1} \mbox{Var}(Y) = \lim\limits_{\kappa \to 1}  \dfrac{1}{\Gamma(\kappa + 1)} = 1,
\end{array}\right. \;\;\;\;\;\; \left\{\begin{array}{ll}
\lim\limits_{\kappa \to 0} \gamma_2 = \lim\limits_{\kappa \to 0}  \dfrac{6 \Gamma^2(\kappa + 1)}{\Gamma(2\kappa + 1)} - 3 = 3,\\ \noalign{\smallskip}
\lim\limits_{\kappa \to 1} \gamma_2 = \lim\limits_{\kappa \to 1}  \dfrac{6 \Gamma^2(\kappa + 1)}{\Gamma(2\kappa + 1)} - 3 = 0.
\end{array}\right.$$

Figure \ref{varecurtose} displays the graphics of the variance and ${\gamma}_2$ as function of $\kappa$. Note that, at the limit of the parametric space of $\kappa$, we have that the variance of the NML distribution is equal to one. For $\kappa>0$, the NML law has tails heavier than the normal distribution since $\gamma_2 > 0$. When $\kappa \to 1$, the excess kurtosis of the Normal Mittag-Leffler distribution coincides with that of the normal law, as expected. The global maximum point of $\Gamma(\kappa + 1)^{-1}$, for $0 < \kappa \leq 1$, is approximately $\kappa = 0.4616$. Therefore, the variance starts at $1$ in the limit when $\kappa \to 0$, increases reaching its maximum value of approximately $1.1292$ and then decreases until it reaches $1$ again when $\kappa \to 1$. The excess kurtosis takes on a maximum value of $3$ when $\kappa \to 0$. This last fact was expected since our model converges to the standard Laplace distribution when $\kappa\rightarrow 0$.

\begin{figure}
	\includegraphics[width=0.49\linewidth]{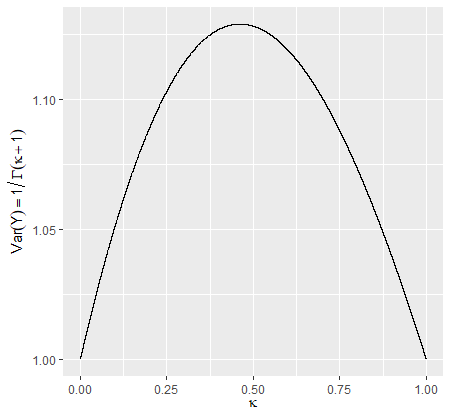} \includegraphics[width=0.49\linewidth]{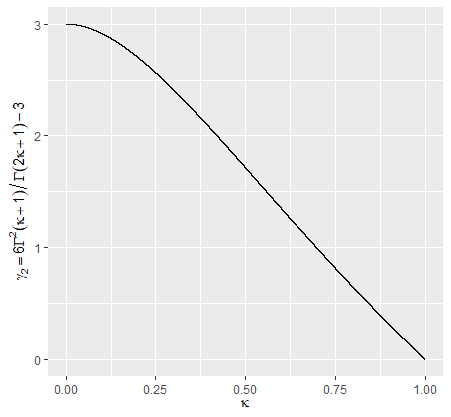}
\caption{Graphics of $\mbox{Var}(Y)$ and $\gamma_2$ as function of $\kappa$.} 
\label{varecurtose}
\end{figure}

From the normal variance mixture representation in (\ref{representacao funcao caracteristica}) and Expression (\ref{densidade mittagleffler}), we obtain that the density function of a NML random variable $Y$ can be expressed by
\begin{eqnarray}\label{densidade NML}
\nonumber f_Y(y) &=& \int\limits_{0}^{\infty} f(y|u)f_{\kappa}(u)du\\
     &=& \int\limits_{0}^{\infty} \frac{1}{\sqrt{2\pi u}}e^{-\frac{y^2}{2u}}\frac{1}{\pi\kappa}\sum\limits_{j = 1}^{\infty}\frac{(-1)^{j-1}}{j!}\sin(\pi\kappa j)\Gamma(\kappa j + 1)u^{j-1} du, \; y \in \mathbb{R}.
\end{eqnarray}

The implementation of the density function above can suffer from numerical instability due to the ML density. To overcome this issue, we find another expression for the NML law based on the Inversion Formula Theorem, which is given in the next proposition.

\begin{proposition}\label{densidade pela formula inversao}
The probability density function of a random variable $Y \sim\mbox{NML}(\kappa)$, for $0 < \kappa \leq 1$, can be expressed by
\begin{eqnarray}\label{densityNML}
f_Y(y) = \frac{1}{\pi}\int\limits_{0}^{\infty}\cos(ty)\mathcal{E}_{\kappa}\left(-\frac{t^2}{2}\right)dt,\;\; y \in \mathbb{R}.
\end{eqnarray}
\end{proposition}

\begin{proof}
To employ the Inversion Formula Theorem (for instance, see Chapter 4 of \cite{g}), we need to show that the characteristic function is integrable, that is $\displaystyle\int\limits_{-\infty}^{\infty}|\varphi_Y(t)|dt < \infty$.

From \cite{aa}, we have that
$\displaystyle\int_{-\infty}^{\infty}\mathcal{E}_{\kappa}\left(-\frac{t^2}{2}\right)dt =  \dfrac{\pi\sqrt{2}}{\Gamma\left(1 - \frac{\kappa}{2}\right)}, \;\; 0 < \kappa \leq 1$. Further, from \cite{pol1948}, we obtain that $\mathcal{E}_{\kappa}\left(-\frac{t^2}{2}\right)$ is completely monotonic and strictly positive for all $t\in\mathbb R$. Using these results, it follows that
$$\int\limits_{-\infty}^{\infty}|\varphi_Y(t)|dt = \int_{-\infty}^{\infty}\left|\mathcal{E}_{\kappa}\left(-\frac{t^2}{2}\right)\right|dt = \int_{-\infty}^{\infty}\mathcal{E}_{\kappa}\left(-\frac{t^2}{2}\right)dt= \frac{\pi\sqrt{2}}{\Gamma\left(1 - \dfrac{\kappa}{2}\right)} < \infty.$$

Finally, apply the Inversion Formula Theorem to obtain that the probability density function of $Y$ can be expressed by
\begin{eqnarray*}
f_Y(y) &=& \frac{1}{2\pi}\int\limits_{-\infty}^{\infty}e^{-ity}\varphi_Y(t)dt
 = \frac{1}{2\pi}\int\limits_{-\infty}^{\infty}e^{-ity}\varphi_Y(t)dt = \frac{1}{2\pi}\int\limits_{-\infty}^{\infty}[\cos(ty) - i\sin(ty)]\varphi_Y(t)dt\\
&=& \frac{1}{2\pi}\int\limits_{-\infty}^{\infty}\underbrace{\cos(ty)\mathcal{E}_{\kappa}\left(-\frac{t^2}{2}\right)}_{\textrm{even function}}dt - \frac{i}{2\pi}\int\limits_{-\infty}^{\infty}\underbrace{\sin(ty)\mathcal{E}_{\kappa}\left(-\frac{t^2}{2}\right)}_{\textrm{odd function}}dt\\
&=& \frac{1}{\pi}\int\limits_{0}^{\infty}\cos(ty)\mathcal{E}_{\kappa}\left(-\frac{t^2}{2}\right)dt,\;\; y \in \mathbb{R}.
\end{eqnarray*}
\end{proof}

Figure \ref{densidadeNML} displays the graph of the probability density function  given in (\ref{densityNML}) for some values of $\kappa$, including the limiting cases Laplace ($\kappa\rightarrow0^+$) and normal ($\kappa\rightarrow1^-$) densities.

\begin{figure}
    \centering
     \includegraphics[width=\linewidth]{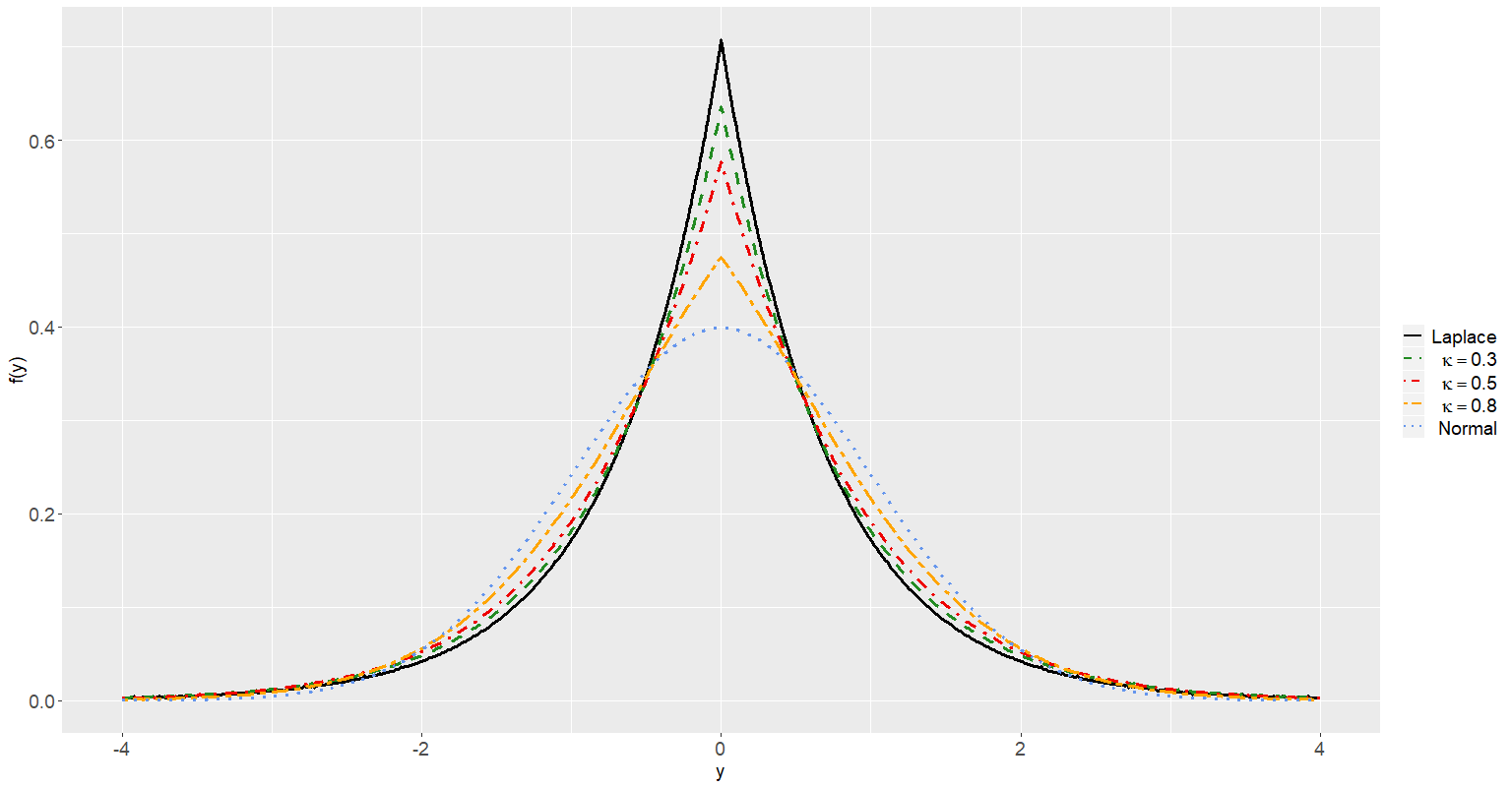}    
    \caption{Graph of the probability density function for $Y \sim\mbox{NML}(\kappa)$ for some values of $\kappa$.}
    \label{densidadeNML}
\end{figure}

We now finish this section by proposing a location-scale extension of our standard NML distribution through a simple linear transformation, which is of practical interest. 

\begin{definition}
Let $Y\sim\mbox{NML}(\kappa)$, $\mu\in\mathbb{R}$ and $\sigma^2>0$. If $X=\sigma Y+\mu$, then we say that $X$ follows a non-standard NML law and denote $X\sim\mbox{NML}(\mu,\sigma^2,\kappa)$.
\end{definition}

From the moments obtained for the standard NML model, we directly obtain the moments of $X\sim\mbox{NML}(\mu,\sigma^2,\kappa)$ as follows:
\begin{eqnarray}\label{momentos nml geral expressao}
E(X^n) = \left\{\begin{array}{lr}
\sum\limits_{j = 0}^{(n-1)/2}\dfrac{n!}{(2j + 1)!}\dfrac{2^{(2j+1-n)/2}}{\Gamma\left(\left(\frac{n-2j-1}{2}\right)\kappa + 1\right)}\sigma^{n-2j-1}\mu^{2j+1},&\mbox{if $n$ is odd},\\[0.7cm]
\sum\limits_{j = 0}^{n/2}\dfrac{n!}{(2j)!}\dfrac{2^{j - n/2}}{\Gamma\left(\left(\frac{n}{2} - j\right)\kappa + 1\right)}\sigma^{n-2j}\mu^{2j}, &\mbox{if $n$ is even}.
\end{array}\right. 
\end{eqnarray} 

In particular, we have that $E(X)=\mu$ and $\mbox{Var}(X)=\dfrac{\sigma^2}{\Gamma(\kappa + 1)}$.
Moreover, from Proposition  \ref{densidade pela formula inversao}, we obtain that the density function of 
$X \sim\mbox{NML}(\mu, \sigma^2, \kappa)$ can be written as
\begin{equation*}\label{densidade NML inversao geral}
    f_X(x) = \dfrac{1}{\sigma\pi}\int\limits_{0}^{\infty}\cos\left(t\left(\dfrac{x-\mu}{\sigma}\right)\right)\mathcal{E}_{\kappa}\left(-\dfrac{t^2}{2}\right)dt, \; \;x \in \mathbb{R}.
\end{equation*}

In the next section, we develop a parameter estimation procedure for the non-standard NML distribution and obtain the asymptotic distribution of the proposed estimators.

\section{Parameter estimation and asymptotic distribution}\label{sec:estimation}

In this section we approach the inferential aspects of the NML distribution. Due to the complicated form of the density function (\ref{densidade NML}) (or (\ref{densityNML})), the maximum likelihood method is cumbersome. A simple widely employed strategy in these cases is to consider a type-method of moments estimation procedure; for instance, see \cite{koz}, \cite{cahetal2010}, \cite{wanetal2014}, and \cite{cahpol2014}. Let $Y_1,\ldots,Y_n$ be an $i.i.d.$ sample from the $\mbox{NML}(\mu, \sigma^2, \kappa)$ law. 
We propose a method of moments (MM) estimation based on the first, second and forth moments of the NML distribution; note that the skewness (related to the third moment) equals 0 in our case. Let $\mu_k = E(Y_1^k)$ and $M_k = \dfrac{1}{n}\sum\limits_{i = 1}^nY_i^k$ denote the $k$-th population moment and $k$-th sampling moment, respectively, $k\in\mathbb N$. The MM estimators, say $\widehat\mu$, $\widehat\sigma^2$, and $\widehat\kappa$, are obtained as the solution of the following system of equations:
\begin{empheq}
[left=\empheqlbrace]{align}\begin{array}{rc}\label{sistema}
\widehat\mu_1 =M_1\\ \noalign{\smallskip}
\widehat\mu_2 =M_2\\ \noalign{\smallskip}
\widehat\mu_4 =M_4
\end{array} \Longrightarrow \empheqlbrace\begin{array}{ll}
\widehat\mu   = M_1\\ \noalign{\smallskip}
\dfrac{\widehat\sigma^2}{\Gamma(\kappa + 1)} + \widehat\mu^2 = M_2\\ \noalign{\smallskip}
\dfrac{6(\widehat\sigma^2)^2}{\Gamma(2\widehat\kappa + 1)} + \dfrac{6\widehat\mu^2\widehat\sigma^2}{\Gamma(\widehat\kappa + 1)} + \widehat\mu^4= M_4.
\end{array}
\end{empheq}

To find the solution of \eqref{sistema}, we study the behavior of the function $h: (0, 1] \to \mathbb{R}$ defined by 
\begin{eqnarray*}\label{h til}
h(\kappa) \equiv \frac{\Gamma^2(\kappa + 1)}{\Gamma(2\kappa + 1)}.
\end{eqnarray*}

Applying the logarithm and calculating the first derivative, we obtain that
\begin{equation}\label{derivada log h}
    \frac{d}{d \kappa} \log{h(\kappa)} = \frac{d}{d \kappa} \log\left\{\frac{\Gamma^2(\kappa + 1)}{\Gamma(2\kappa + 1)} \right\} = 2[\Psi(\kappa + 1) - \Psi(2\kappa + 1)],
\end{equation}
where $\Psi(z) = \dfrac{\Gamma'(z)}{\Gamma(z)}$ is the digamma function. On the other hand, $\Psi(z)$ can be written as (see \cite{as})
$$\Psi(z) = \int\limits_{0}^{1}\dfrac{1 - t^{z - 1}}{1 - t}dt - \gamma,\,\,\,\,\,z>0,$$
where $\gamma = \int\limits_{0}^{\infty}\left(\dfrac{1}{e^t - 1} - \dfrac{1}{te^t}\right)dt$. Hence, we can write (\ref{derivada log h}) as
$$\frac{d}{d \kappa} \log{h(\kappa)} = 2\int\limits_{0}^{1}\dfrac{-t^{\kappa}(1 - t^{\kappa})}{1 - t}dt.$$

Since $\int\limits_{0}^{1}\dfrac{-t^{\kappa}(1 - t^{\kappa})}{1 - t}dt < 0$ for all $\kappa \in (0, 1]$, it follows that $$\frac{d}{d \kappa} \log{h(\kappa)} < 0, \,\,\forall \; \; \kappa \in (0, 1].$$

This shows that $h(\cdot)$ is monotone function. Also, the continuity of $h(\cdot)$ follows from the continuity of the gamma function. Therefore, $h(\cdot)$ has an inverse function which is denoted by $h^{-1}(\cdot)$. Figure \ref{sinalh} presents the graphic of  the function $\dfrac{d}{d \kappa} \log{h(\kappa)}$ versus $\kappa$, for $\kappa \in (0, 1]$.

\begin{figure}
    \centering
    \includegraphics[width=0.7\linewidth]{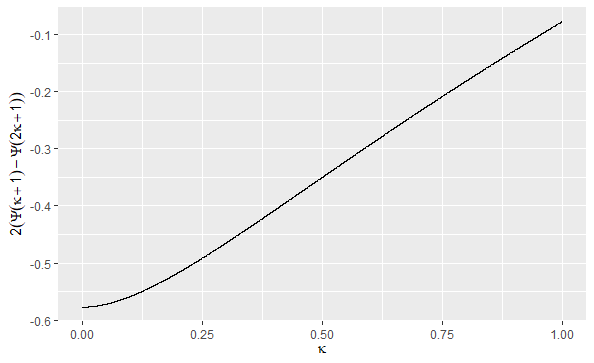}    
    \caption{Graph of $\dfrac{d}{d \kappa} \log{h(\kappa)}$ versus $\kappa$, for $\kappa \in (0, 1]$.}
    \label{sinalh}
\end{figure}

Solving the system of nonlinear equations in (\ref{sistema}), we obtain explicit expressions for the MM estimators as follows:
\begin{eqnarray}\label{estimacao metodo momentos}
\left\{\begin{array}{lll} \widehat{\mu} &=& M_1,\\
{ \widehat{\sigma}}^2 &=& (M_2 - M_1^2)\Gamma\left(h^{-1}\left(\dfrac{M_4 - 6M_1^2M_2 + 5M_1^4}{6(M_2 - M_1^2)^2}\right) + 1\right),\\
 \widehat{\kappa} &=& h^{-1}\left(\dfrac{M_4 - 6M_1^2M_2 + 5M_1^4)}{6(M_2 - M_1^2)^2}\right).
\end{array}\right.
\end{eqnarray}

In the next proposition, we establish the asymptotic properties of the MM estimators given in (\ref{estimacao metodo momentos}).

\begin{proposition}\label{MM_AN}
The MM estimators $\widehat\mu$, $\widehat\sigma^2$, and $\widehat\kappa$ are strongly consistent for $\mu$, $\sigma^2$, and $\kappa$, respectively, and satisfy the asymptotic normality $$\sqrt{n}\left\{( \widehat{\mu}, { \widehat{\sigma}}^2,  \widehat{\kappa}) - (\mu, \sigma^2, \kappa)\right\} \overset{d}{\longrightarrow} N_3(0, \nabla g\Sigma \nabla g^\top),$$
as $n\rightarrow\infty$, where the elements of the matrix $\Sigma$ are
\begin{eqnarray*}
&&\Sigma_{11} = \dfrac{\sigma^2}{\Gamma(\kappa + 1)}, \quad \Sigma_{12} = \Sigma_{21} =  \dfrac{2\mu\sigma^2}{\Gamma(\kappa + 1)},\quad
\Sigma_{13} = \Sigma_{31} = \dfrac{24\mu\sigma^4}{\Gamma(2\kappa + 1)} + \dfrac{4\mu^3\sigma^2}{\Gamma(\kappa + 1)},\\
&&\Sigma_{22} = \dfrac{6\sigma^4}{\Gamma(2\kappa + 1)} + \dfrac{4\mu^2\sigma^2}{\Gamma(\kappa + 1)} - \dfrac{\sigma^4}{[\Gamma(\kappa + 1)]^2},\\
&&\Sigma_{23} = \Sigma_{32} =  \dfrac{90\sigma^6}{\Gamma(3\kappa + 1)} - \dfrac{6\sigma^6}{\Gamma(2\kappa + 1)\Gamma(\kappa + 1)} + \dfrac{84\mu^2\sigma^4}{\Gamma(2\kappa + 1)} - \dfrac{6\mu^2\sigma^4}{[\Gamma(\kappa + 1)]^2} + \dfrac{8\mu^4\sigma^2}{\Gamma(\kappa + 1)},\\
&&\Sigma_{33} =
     \dfrac{16\mu^6\sigma^2}{\Gamma(\kappa + 1)} +  \dfrac{408\mu^4\sigma^4}{\Gamma(2\kappa + 1)} -  \dfrac{36\mu^4\sigma^4}{[\Gamma(\kappa + 1)]^2} -
      \dfrac{72\mu^2\sigma^6}{\Gamma(\kappa + 1)\Gamma(2\kappa + 1)} +
       \dfrac{2520\mu^2\sigma^6}{\Gamma(3\kappa + 1)}\\ &&\hspace{6cm}+ \dfrac{2520\sigma^8}{\Gamma(4\kappa + 1)} - 
         \dfrac{36\sigma^8}{[\Gamma(2\kappa + 1)]^2},
\end{eqnarray*} and
$g: \mathbb{R}^3 \to \mathbb{R}^3$ is defined by
\begin{flalign}\label{funcao g}
\nonumber g(x, y, z) &\equiv (g_1(x, y, z), g_2(x, y, z), g_3(x, y, z))\\ 
&\equiv \left(x, (y - x^2)\Gamma\left(h^{-1}\left(\dfrac{z - 6x^2y + 5x^4}{6(y - x^2)^2}\right) + 1\right), h^{-1}\left(\dfrac{z - 6x^2y + 5x^4}{6(y - x^2)^2}\right)\right),
\end{flalign}
with gradient $\nabla g$ given explicitly in the Appendix.
\end{proposition}

\begin{proof}
The strongly consistency of the MM estimators follows from the Strong Law of Large Numbers. To establish the asymptotic normality, we apply the multivariate Central Limit Theorem to obtain that
\begin{eqnarray}\label{MCLT_app}
\sqrt{n}\left\{\left(\frac{\sum_{i = 1}^{n} Y_i}{n}, \frac{\sum_{i = 1}^{n} Y_i^2}{n}, \frac{\sum_{i = 1}^{n} Y_i^4}{n}\right) - \left(\mu_1, \mu_2, \mu_4\right)\right\} \overset{d}{\longrightarrow} N_3(0, \Sigma),\quad \mbox{as}\quad n\rightarrow\infty,
\end{eqnarray}
where $\mu_1 = E(Y) = \mu$, $\mu_2 = E(Y^2) = \dfrac{\sigma^2}{\Gamma(\kappa + 1)} + \mu^2$, $\mu_4 = E(Y^4) = \dfrac{6\mu^2\sigma^2}{\Gamma(\kappa + 1)} + \dfrac{6\sigma^4}{\Gamma(2\kappa + 1)} + \mu^4$, and $\Sigma$ is the asymptotic covariance matrix given by
\begin{eqnarray*}
\Sigma = 
\begin{pmatrix}
\mbox{Var}(Y)&\mbox{cov}(Y, Y^2)&\mbox{cov}(Y, Y^4)\\
\mbox{cov}(Y, Y^2)&\mbox{Var}(Y_1^2)&\mbox{cov}(Y^2, Y^4)\\
\mbox{cov}(Y, Y^4)&\mbox{cov}(Y^2, Y^4)&\mbox{Var}(Y^4)
\end{pmatrix},
\end{eqnarray*}
with explicit elements obtained from the moments of the random variable $Y\sim\mbox{NML}(\mu,\sigma^2,\kappa)$ given in (\ref{momentos nml geral expressao}). The desired result is now obtained by applying the Delta Method in (\ref{MCLT_app}) with $g(\cdot,\cdot,\cdot)$ assuming the form (\ref{funcao g}).
\end{proof}

In the next section, we provide numerical experiments involving artificial and real data sets to illustrate the finite-sample performance of the proposed estimators and the usefulness of the NML law in practice. 

\section{Numerical experiments}\label{sec:sim_app}

\subsection{Monte Carlo simulation}\label{sec:simulation}%\label{simulacao NML}

We present a simulation study to verify the behavior of the proposed estimators for the NML parameters. All the numerical experiments provided in this paper were implemented using the software \texttt{R} \citep{R}. To generate random samples from the NML distribution, we use its normal variance mixture representation given in (\ref{representacao funcao caracteristica}), that is, $Y = \mu + \sigma\sqrt{U}Z$, with $Z \sim N(0, 1)$ independent of $U \sim ML(\kappa)$. To generate from the ML distribution, we consider the algorithm proposed by \cite{ri}, which is based on the Laplace transform.

In these simulations, we set $\mu = 0.5$, $\sigma^2 = 1$, $\kappa = 0.2,\; 0.3,\; 0.5,\; 0.6,\;0.8$, and sample sizes $n = 200, 500, 1000, 2000$. We also set 5000 Monte Carlo replications and in each step we generate a random sample from the NML distribution and obtain the MM estimates. The empirical means of the parameter estimates and their root mean square error (RMSE) are reported in Table \ref{resultados simulacao NML todos cenarios}.

\begin{table}
\scriptsize
\centering
\caption{Empirical mean and root mean square error (in parentheses) of the
parameter estimates under the NML model with $\mu=0.5$, $\sigma^2=1$, $\kappa = 0.2,\; 0.3,\; 0.5,\; 0.6,\;0.8$, and sample sizes $n = 200, 500, 1000, 2000$.}
\label{resultados simulacao NML todos cenarios}
\begin{tabular}{ccccc}
\toprule
 & \multicolumn{4}{c}{Estimates (RMSE)} \\[2pt]
 &\multicolumn{1}{c}{$n = 200$} & \multicolumn{1}{c}{$n = 500$} & \multicolumn{1}{c}{$n = 1000$} & \multicolumn{1}{c}{$n = 2000$} \\[2pt]\midrule
  
$\mu = 0.5$      &0.4920 (0.0729) &0.4958 (0.0460) &0.4972 (0.0330) &0.4983 (0.0235)\\[2pt]
$\sigma^2 = 1.0$ &0.9615 (0.1469) &0.9657 (0.0999) &0.9717 (0.0781) &0.9790 (0.0596)\\[2pt]
$\kappa = 0.2$   &0.5138 (0.3809) &0.4261 (0.2843) &0.3557 (0.2135) &0.3056 (0.1631)\\[12pt] 

$\mu = 0.5$      &0.4952 (0.0732) &0.4958 (0.0472) &0.4980 (0.0330) &0.4992 (0.0237)\\[2pt]
$\sigma^2 = 1.0$ &0.9956 (0.1451) &0.9899 (0.0963) &0.9929 (0.0698) &0.9974 (0.0532)\\[2pt]
$\kappa = 0.3$   &0.5390 (0.3205) &0.4454 (0.2295) &0.3952 (0.1781) &0.3517 (0.1368)\\[12pt] 

$\mu = 0.5$      &0.4968 (0.0753) &0.4991 (0.0471) &0.4989 (0.0332) &0.4999 (0.0239)\\[2pt]
$\sigma^2 = 1.0$ &1.0188 (0.1371) &1.0145 (0.0880) &1.0084 (0.0620) &1.0069 (0.0442)\\[2pt]
$\kappa = 0.5$   &0.5998 (0.2305) &0.5471 (0.1846) &0.5201 (0.1479) &0.5102 (0.1148)\\[12pt] 

$\mu = 0.5$      &0.4997 (0.0732) &0.4994 (0.0477) &0.5006 (0.0336) &0.5003 (0.0231)\\[2pt]
$\sigma^2 = 1.0$ &1.0189 (0.1331) &1.0139 (0.0851) &1.0075 (0.0591) &1.0043 (0.0407)\\[2pt]
$\kappa = 0.6$   &0.6559 (0.2101) &0.6281 (0.1725) &0.6088 (0.1364) &0.6030 (0.0986)\\[12pt] 

$\mu = 0.5$      &0.5006 (0.0733) &0.5003 (0.0462) &0.5011 (0.0325) &0.4997 (0.0233)\\[2pt]
$\sigma^2 = 1.0$ &1.0034 (0.1199) &1.0029 (0.0787) &1.0021 (0.0565) &1.0017 (0.0402)\\[2pt]
$\kappa = 0.8$   &0.7690 (0.1692) &0.7955 (0.1258) &0.8010 (0.0943) &0.8032 (0.0695)\\\bottomrule    
\end{tabular}
\end{table}

Looking at the results from Table \ref{resultados simulacao NML todos cenarios}, we observe that the bias and RMSE go to 0 as the sample size increases. This was expected since the estimators are consistent. In particular, the MM estimators of $\mu$ and $\sigma^2$ work very well in all scenarios considered. Regarding the estimation of $\kappa$, we see that $\widehat\kappa$ yields satisfactory estimates for $\kappa \geq 0.5$, but not for $\kappa=0.2,0.3$. Note that there is a considerable bias when  $\kappa<0.5$, even for large samples. The estimation under this setting needs further investigation. It is worth anticipating that this type of problem is not experienced in our application since $\widehat\kappa\approx0.5$ there.

A second simulation study enables us to evaluate the standard errors obtained through the asymptotic covariance matrix given in Proposition \ref{MM_AN}, which we call here the theoretical standard error. We compare them to the empirical standard errors obtained from the MM estimates in the same settings as before. The average theoretical and empirical standard errors are presented in Table \ref{resultados erro padrao NML todos cenarios}. We observe a good agreement between the theoretical and empirical standard errors for the cases where $\kappa\geq0.5$, especially when the sample size increases. On the other hand, we notice a considerable difference for the settings $\kappa=0.2$ and $\kappa=0.3$, except for the standard errors related to $\widehat\mu$. Again, this gives evidence that a special study is required to make an inference on the NML law when $\kappa<0.5$.

\begin{table}
\scriptsize
\centering
\caption{Average theoretical and empirical standard errors of the parameter estimates under the NML model with $\mu=0.5$, $\sigma^2=1$, $\kappa = 0.2,\; 0.3,\; 0.5,\; 0.6,\;0.8$, and sample sizes $n = 200, 500, 1000, 2000$.}
\label{resultados erro padrao NML todos cenarios}
\resizebox{\columnwidth}{!}{%
\begin{tabular}{cccccccccc}
\toprule
& & \multicolumn{8}{c}{Standard errors} \\[2pt]
 &   & \multicolumn{2}{c}{$n = 200$} &\multicolumn{2}{c}{$n = 500$} &\multicolumn{2}{c}{$n = 1000$} &\multicolumn{2}{c}{$n = 2000$} \\[2pt]
& &Empirical&Theoretical&Empirical&Theoretical&Empirical&Theoretical&Empirical&Theoretical\\[2pt]\midrule
& $\mu$ &0.0725 & 0.0726 & 0.0458 & 0.0463 & 0.0329 & 0.0328 & 0.0235 & 0.0233 \\[2pt]
$\kappa = 0.2$&$\sigma^2$&0.1418 & 0.2064 & 0.0939 & 0.1397 & 0.0728 & 0.1174 & 0.0558 & 0.1046 \\[2pt]
& $\kappa$&0.2159 & 0.5055 & 0.1724 & 0.3648 & 0.1461 & 0.3069 & 0.1244 & 0.2619\\[12pt]

& $\mu$ & 0.0731 & 0.0739 & 0.0470 & 0.0468 & 0.0330 & 0.0332 & 0.0237 & 0.0235 \\[2pt]
$\kappa = 0.3$&$\sigma^2$& 0.1450 & 0.1971 & 0.0957  & 0.1462  & 0.0694  & 0.1059  & 0.0532 & 0.0788 \\[2pt]
 & $\kappa$ & 0.2134 & 0.4644 & 0.1775 & 0.3593 & 0.1506 & 0.2730 & 0.1266 & 0.2079\\[12pt]

&$\mu$&0.0752 & 0.0746 & 0.0471 & 0.0474 & 0.0332 & 0.0335 &0.0239  &0.0237 \\[2pt]
$\kappa = 0.5$&$\sigma^2$&0.1358 & 0.1805 & 0.0868 & 0.1117 & 0.0614 & 0.0745&0.0437  & 0.0494\\[2pt]
&$\kappa$& 0.2079 & 0.3991 & 0.1785 & 0.2646 & 0.1466 & 0.1868& 0.1144 & 0.1287\\[12pt]

&$\mu$&0.0732 & 0.0744 & 0.0477 & 0.0472 & 0.0336 & 0.0334 & 0.0231 & 0.0236 \\[2pt]
$\kappa = 0.6$&$\sigma^2$&0.1318 & 0.1606 & 0.0840 & 0.0988 & 0.0586 & 0.0655 & 0.0404 & 0.0445 \\[2pt]
&$\kappa$&0.2026 & 0.3409 & 0.1702 & 0.2169 & 0.1362 & 0.1501 & 0.0986 & 0.1033\\[12pt]

&$\mu$&0.0733 & 0.0732 & 0.0462 & 0.0463 & 0.0325 & 0.0327 & 0.0233 & 0.0232 \\[2pt]
$\kappa = 0.8$&$\sigma^2$&0.1199 & 0.1428 & 0.0786 & 0.0874 & 0.0565 & 0.0613 & 0.0402 & 0.0433 \\[2pt]
&$\kappa$&0.1663 & 0.2610 & 0.1257 & 0.1499 & 0.0943 & 0.1029 & 0.0694 & 0.0717\\\bottomrule
\end{tabular}%
}
\end{table}

We conclude this subsection with an illustration of the asymptotic normality of the MM estimators established in Proposition \ref{MM_AN}. We present histograms of the standardized (mean-deviation) MM estimates obtained in the Monte Carlo replications along with the standard normal density curve in Figure \ref{hist_AN} for the cases $\kappa = 0.3$ and $\kappa = 0.8$ with sample sizes $n=200,500,1000,2000$. It can be seen that the normal approximation works satisfactorily and becomes better as the sample size increases. We also observed this behavior for the MM estimators of $\mu$ and $\sigma^2$ but we did not report them here to save space.

\begin{figure}
\center
%\subfigure{\includegraphics[width=\linewidth]{tese_EstaPasta/boxplotkappa0_3.png}}
\subfigure{\includegraphics[width=\linewidth, height = 0.27\columnwidth]{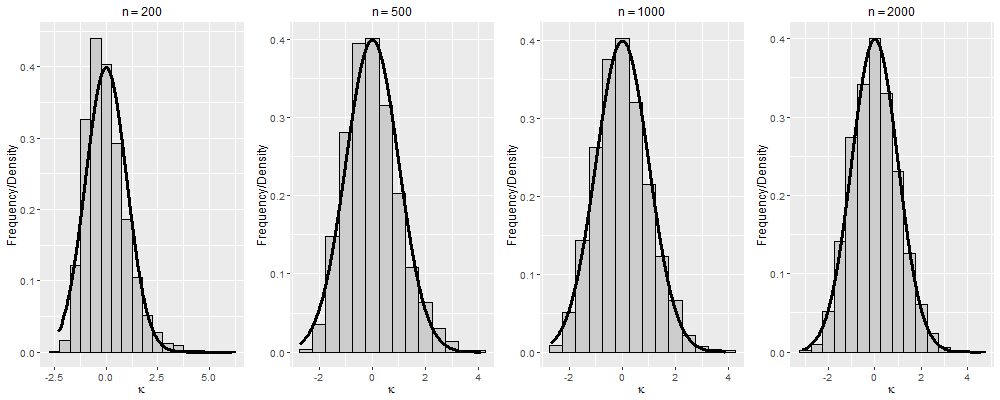}}\\
\subfigure{\includegraphics[width=\linewidth, height = 0.27\columnwidth]{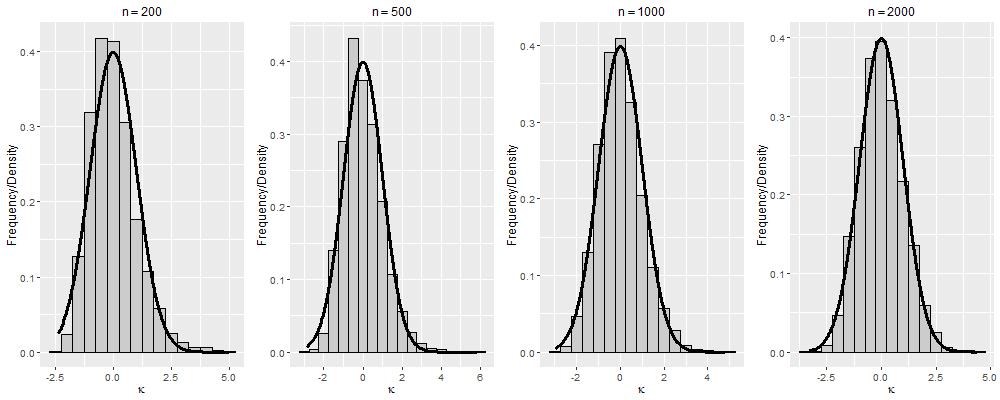}}
\caption{Histograms of the standardized MM estimates of $\kappa$ under the $\mathrm{NML}(\mu = 0.5, \sigma^2 = 1.0, \kappa)$ model for $\kappa=0.3$ in the top row and $\kappa=0.8$ in the bottom row, with sample sizes $n=200,500,1000,2000$.}
\label{hist_AN}
\end{figure}

\subsection{IBOVESPA data analysis}\label{sec:application}

We now illustrate the usefulness of the Normal-Mittag-Leffler distribution for modeling financial data. We study the daily log-returns of \texttt{IBOVESPA} (S\~ao Paulo Stock Exchange Index), which can be obtained from the website \url{https://finance.yahoo.com/}. The data consists of 2226 observations collected from Jan $1^{st}$, 2010 to Dec $31^{th}$, 2018. For a justification of why using limiting laws from random summations to model high-frequency financial data, we refer to \cite{st} and \cite{olietal2020}.

For comparison purposes, we also consider the well-known normal inverse-Gaussian (NIG) and normal gamma (NG) distributions with parameters $\mu\in\mathbb R$, $\sigma^2>0$, and $\phi>0$; we use here the parameterization considered in \cite{olietal2020}. For estimating their parameters, we consider the EM-algorithm developed by \cite{olietal2020}. Moreover, the limiting cases of the NML distribution, normal ($\kappa=1$) and Laplace ($\kappa=0$), are also considered in our application.

Table \ref{resultados ibovespa} presents the parameter estimates and their associated standard errors for the NML, normal, Laplace, NG, and NIG models for the daily log-returns of \texttt{IBOVESPA}. In Figure \ref{histogramIbovespaIndex}, we show the histogram of the \texttt{IBOVESPA} log-returns along with the normal, Laplace, NML, NG, and NIG fitted density functions. From this figure, we observe a satisfactory fit of our proposed NML distribution to the data. To confirm this, we compute the empirical first four cumulants and compare them to the fitted quantities according to the models considered in this application. These empirical and fitted cumulants are provided in Table \ref{Descriptive statistics Ibovespa}. Looking at this table, we see that all models captured well the empirical mean and variance, except for the estimated variance under the Laplace model. The empirical skewness is almost null, so the assumption of symmetry implicitly considered in the normal, Laplace, and NML laws is suitable. Moreover, the NG and NIG distributions estimated the skewness close to 0.

\begin{table}
\small
\centering
\vspace{0.5cm}
\caption{Parameter estimates and their standard errors (in parentheses) for the NML, normal, Laplace, NG, and NIG models for the daily log-returns of \texttt{IBOVESPA}.}
\label{resultados ibovespa}
\begin{tabular}{lccc}
\toprule
\multicolumn{1}{c}{Model} & $\mu$       &$\sigma^2$      & $\kappa$ or $\phi$ \\[2pt]\midrule
Normal &0.00021 (0.00030)    & 0.00020 (0.00001)  & 1 (-------)              \\[2pt]
Laplace&0.00021 (0.00012)    &0.00021 (0.00001)   & 0 (-------)              \\[2pt]
NML    &0.00021 (0.00030)    &0.00018 (0.00001)   & 0.49123 (0.00554)   \\[2pt]
NIG    &0.00021  (0.00030)   &0.00020 (0.00001)   & 2.09675 (0.41203)\\[2pt]
NG     &0.00021 (0.00030)    &0.00020 (0.00001)   &2.58371 (0.40600)\\\bottomrule
\end{tabular}
\end{table}

The major difference in the fitted models concerns the tails through the excess kurtosis. First, note that the normal and Laplace distributions have theoretical excess kurtosis equal to 0 and 3, respectively. These distributions are not adequate to model the tails in this financial data set. We observe that our NML distribution captured very well the excess kurtosis, providing better results than the well-used NIG and NG laws. It is important to emphasize the cruciality of modeling well the tails when dealing with financial data. This empirical illustration shows that the limiting distribution found in our paper can be useful for analyzing data when the excess kurtosis is between 0 and 3.

\begin{figure}[!htb]
\centering
\includegraphics[width=.8\linewidth]{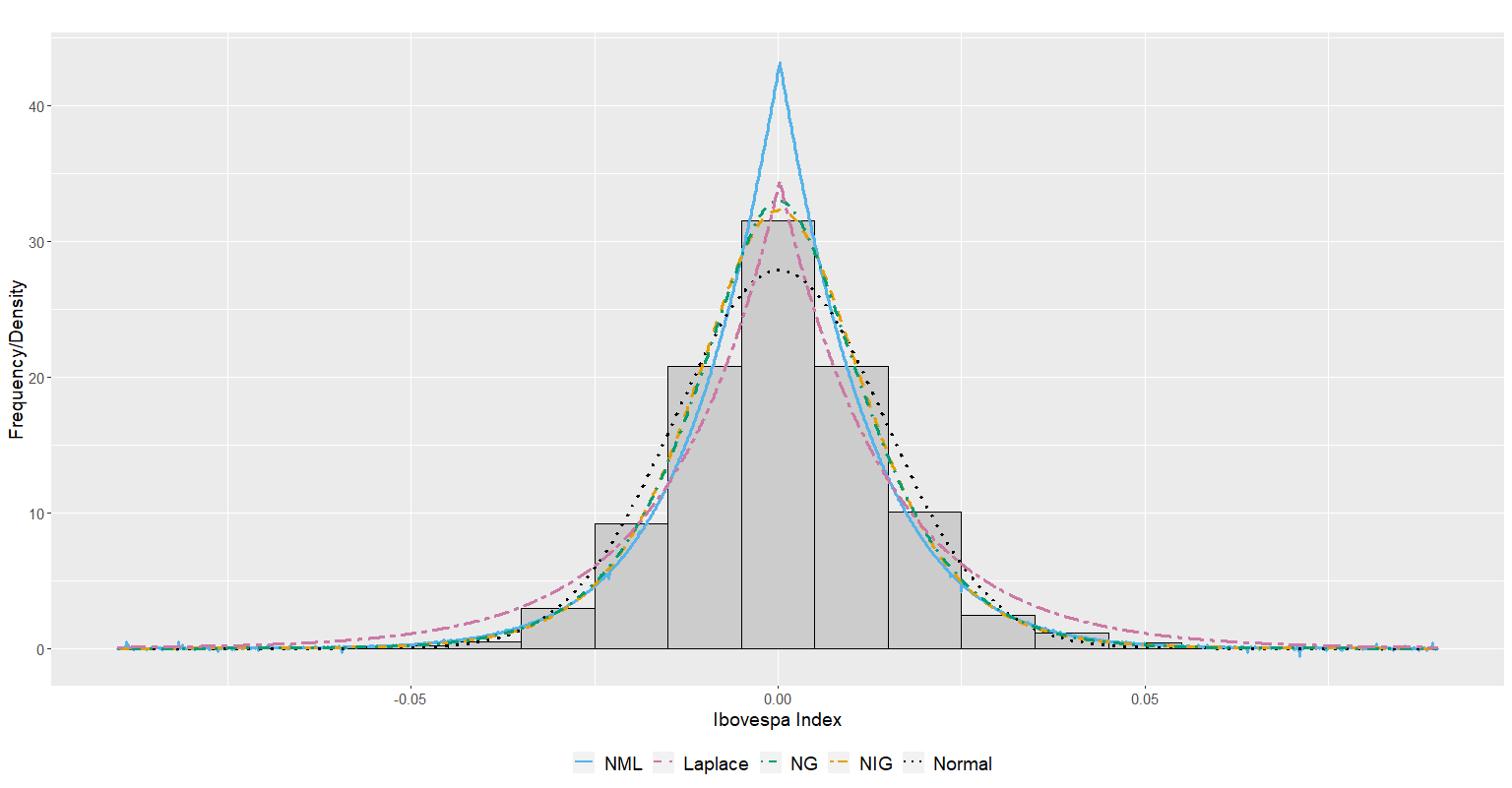}
\caption{Histogram of the daily log-returns of \texttt{IBOVESPA} along with the normal, Laplace, NML, NG, and NIG fitted density functions.}
\label{histogramIbovespaIndex}
\end{figure}

\begin{table}
\small
\centering
\caption{Empirical and fitted mean, variance, skewness, and excess kurtosis for the daily log-returns of \texttt{IBOVESPA} under the normal, Laplace, NML, NG, and NIG laws.}
\label{Descriptive statistics Ibovespa}
\begin{tabular}{lcccc}
\toprule
& Mean     &Variance & Skewness &Excess kurtosis    \\[2pt]\midrule
Normal    &0.00021 & 0.00020 & 0 & 0\\[2pt]
Laplace    &0.00021 & 0.00042 & 0 &3 \\[2pt]
NIG    &0.00021&0.00020  &0.05172 &1.35247\\[2pt]
NG    &0.00020 &0.00020 &0.02624 &0.47885\\[2pt]
NML   &0.00021 &0.00020  & 0 &1.74430 \\[2pt]\hdashline
Empirical & 0.00021 &0.00020 &$-$0.05782 &1.74700 \\[2pt]\bottomrule
\end{tabular}
\end{table}

\section{Related problems}\label{sec:relatedproblems}

In this section, we discuss two related problems to the FP summation: (1) the mixed Poisson representation of the FP distribution; (2) the weak limit of a Conway-Maxwell-Poisson random sum.

\subsection{Mixed Poisson representation for the FP law}

Proposition \ref{convergencia distribuicao padrao ML} provides, as a byproduct, the decomposition of an FP distribution as a mixture between the Poisson and (type 2) Mittag-Leffler distributions, which is stated in the next corollary. To the best of our knowledge, this property of the FP law is a new finding in the literature.

\begin{corollary}\label{MP_rep}
Let $N\sim{FP}(\nu,\kappa)$, $\nu>0$ and $\kappa\in(0,1]$. Then, $N$ satisfies the following mixed Poisson representation: $N|U\sim\mbox{Poisson}(\nu U)$, with $U\sim\mbox{ML}(\kappa)$.
\end{corollary}

We think that Corollary \ref{MP_rep} can be useful to study properties of the FP distribution and to implement alternative estimation procedures rather than the method of moments. In particular, a Monte Carlo Expectation-Maximization algorithm can be developed as an alternative to the direct maximization of the log-likelihood function, which is infeasible due to the complicated form of the FP probability function.

Moreover, by using Corollary \ref{MP_rep} and Definition 4.3 from \cite{gra1997}, we can construct an associated mixed Poisson process as follows. Let $\{N^\star(t)\}_{t\geq0}$ be a Poisson process with rate $\nu>0$, independent of $U\sim\mbox{ML}(\kappa)$. Define a new counting process by $\{\widetilde N(t)\}_{ t\geq0}$ by $\widetilde N(t)\equiv (N^\star \circ U) (t)\equiv N^\star (t U)$, for $t\geq0$. Then, $\{\widetilde N(t)\}_{t\geq0}$ is a mixed Poisson process with marginals FP distributed.

We believe that the points discussed in this subsection deserve further investigation. Another question of interest is to explore if the existing non-markovian FP process is related (in some sense) to the above mixed Poisson process.

\subsection{Conway-Maxwell-Poisson random sum}

We conclude this paper showing that not all generalized Poisson random sums yield a non-normal limiting distribution. We consider the  Conway-Maxwell-Poisson (COMP) distribution, also known as COM-Poisson, introduced by \cite{cm}. This distribution has received a lot of attention after its revival by \cite{smkbb}.

A random variable $K$ is COM-Poisson distributed if its probability mass function assumes the form
$$P(K = j) = \frac{\lambda^j}{(j!)^{\eta}}\frac{1}{H(\lambda, \eta)}, \quad j = 0, 1, 2, \cdots,$$
for $\lambda>0$ and $\eta > 0$, where $H(\lambda, \eta)$ is a normalizing constant given by
$$H(\lambda, \eta) = \sum\limits_{i = 0}^{\infty}\frac{\lambda^i}{(i!)^{\eta}}.$$
We denote $K \sim\mbox{COMP}(\lambda, \eta)$. 
The COM-Poisson law contains the Poisson distribution as a particular case when $\eta = 1$. The expected value, variance, and probability generating function of $K \sim\mbox{COMP}(\lambda, \eta)$ are (for instance, see \cite{dg}) respectively
\begin{eqnarray*}\label{expVar}
E(K) = \lambda\frac{d}{d\lambda}\bigg\{\log(H(\lambda, \eta))\bigg\}, \; \; \mbox{Var}(K) = \lambda\frac{d}{d\lambda}\bigg\{E(K)\bigg\}, \; \; G_{K}(t) = \frac{H(t\lambda, \eta)}{H(\lambda, \eta)},\,\, t\in\mathbb R.
\end{eqnarray*}

In \cite{g1}, it is shown that, for fixed $\eta$ and for large $\lambda$,
\begin{equation}\label{aproxCOMPoisson}
    H(\lambda, \eta) = \dfrac{\exp\{\eta\lambda^{1/\eta}\}}{\lambda^{\frac{\eta-1}{2\eta}}(2\pi)^{\frac{\eta-1}{2}}\sqrt{\eta}}(1+\mathcal{O}(\lambda^{-1/\eta})),
\end{equation}
and
\begin{eqnarray}\label{aproxExpVar}
E(K) \approx \lambda^{\frac{1}{\eta}} - \frac{\eta-1}{2\eta}.
\end{eqnarray}

The Conway-Maxwell-Poisson random sum is given by (\ref{random_sum}) with $N=K \sim\mbox{COMP}(\lambda, \eta)$. When $\eta=1$, we obtain the Poisson random summation as a particular case. In the next proposition, we show that a proper normalization of the COMP summation yields a normal limiting distribution.

\begin{proposition}\label{padrao_COMPoisson}
Let $\{X_n\}_{n \in \mathbb{N}}$ be a sequence of i.i.d. random variables with $E(X_1) = 0$ and $\mbox{Var}(X_1)=1$. Let $K \sim COMP(\lambda, \eta)$, independent of the $X_n's$. Then, 

\begin{equation*}
 \widetilde{S}_{\lambda} \equiv a_{\lambda}\sum\limits_{i = 1}^{K} X_i \overset{d}\longrightarrow N(0,1), \; \; as\, \lambda \to \infty,
\end{equation*}
where $a_{\lambda} = \lambda^{-\frac{1}{2\eta}}.$
\end{proposition}

\begin{proof}
Let $S_{\lambda}=\sum\limits_{i = 1}^{K} X_i$ and observe that $E(S_{\lambda}) = E(K)E(X_1) = 0$. Also, with the help of (\ref{aproxExpVar}), we get $\mbox{Var}(S_{\lambda}) = \mbox{Var}(X_1)E(K) + E^2(X_1)\mbox{Var}(K) \approx \lambda^{\frac{1}{\eta}} - \frac{\eta-1}{2\eta}$, when $\lambda$ is large. Therefore, $\mbox{Var}(S_{\lambda}) = \mathcal O(\lambda^{1/\eta})$, which indicates we must choose $a_{\lambda}=\lambda^{-\frac{1}{2\eta}}$ as the proper normalization.

The characteristic function of the random variable $\widetilde{S}_{\lambda}$ can be written as
\begin{eqnarray*}
\varphi_{\widetilde{S}_{\lambda}}(t) = \varphi_{S_{\lambda}}(a_{\lambda}t)
= G_K\left(\varphi_{X_1}(a_{\lambda}t)\right)
= \frac{H\left(\varphi_{{X_1}}(a_{\lambda}t)\lambda, \eta\right)}{H\left(\lambda, \eta\right)},
\end{eqnarray*}
with $G_K(\cdot)$ and $\varphi_{{X_1}}(\cdot)$ being the probability generating function of $K$ and characteristic function of $X_1$, respectively. Taking $\lambda\rightarrow\infty$ and using (\ref{aproxCOMPoisson}), we obtain
\begin{eqnarray}\label{lim_aux3}
\nonumber \lim\limits_{\lambda \to \infty} \varphi_{\widetilde{S}_{\lambda}}(t) &=& \lim\limits_{\lambda \to \infty} \dfrac{\exp\left\{\eta\lambda^{1/\eta}\left[\varphi^{1/\eta}_{{X_1}}(a_{\lambda}t) - 1\right] \right\}}{\varphi_{{X_1}}^{(\eta - 1)/2\eta}\left(a_{\lambda t}\right)}
= \nonumber \dfrac{\exp\left\{\lim\limits_{\lambda \to \infty}\eta\lambda^{1/\eta}\left[\varphi^{1/\eta}_{{X_1}}(a_{\lambda}t) - 1\right] \right\}}{\lim\limits_{\lambda \to \infty} \varphi_{{X_1}}^{(\eta - 1)/2\eta}\left(a_{\lambda t}\right)}\\
&=&\exp\left\{\lim\limits_{\lambda \to \infty}\frac{\eta\left[\varphi^{1/\eta}_{{X_1}}(a_{\lambda}t) - 1\right] }{\lambda^{-1/\eta}}\right\}.
\end{eqnarray}

Note that $\lim\limits_{\lambda \to \infty} \varphi^{1/\eta}_{{X_1}}(a_{\lambda}t) = 1$ and, therefore, the limit in (\ref{lim_aux3}) has the indeterminate form ``0/0". Apply L'H\^opital's rule in (\ref{lim_aux3}) to obtain that
\begin{eqnarray}\label{lim_aux4}
 \lim\limits_{\lambda \to \infty}\frac{\eta\left[\varphi^{1/\eta}_{{X_1}}(a_{\lambda}t) - 1\right] }{\lambda^{-1/\eta}} &=& \frac{t}{2}\lim\limits_{\lambda \to \infty}\frac{\varphi^{1/\eta-1}_{{X_1}}(\lambda^{-\frac{1}{2\eta}}t)\varphi'_{X_1}(\lambda^{-\frac{1}{2\eta}}t)}{\lambda^{-\frac{1}{2\eta}}},
\end{eqnarray}
where $\varphi'_{X_1}(x)=d\varphi_{X_1}(x)/dx$. Note that Expression (\ref{lim_aux4}) has again the indeterminate form ``0/0" since $\varphi'_{X_1}(0)=iE({X_1})=0$. A second application of the L'H\^opital's rule gives us that
\begin{eqnarray}\label{lim_aux5}
&&\displaystyle\lim_{\lambda \to \infty}\frac{\varphi^{1/\eta-1}_{{X_1}}(\lambda^{-\frac{1}{2\eta}}t)\varphi'_{X_1}(\lambda^{-\frac{1}{2\eta}}t)}{\lambda^{-\frac{1}{2\eta}}}=\nonumber\\
\nonumber&& \lim\limits_{\lambda \to \infty}\left\{\left(1/\eta - 1\right)\varphi^{1/\eta-2}_{{X_1}}(\lambda^{-\frac{1}{2\eta}}t)\left[\varphi'_{{X_1}}(\lambda^{-\frac{1}{2\eta}}t)\right]^2t + \varphi^{1/\eta-1}_{{X_1}}(\lambda^{-\frac{1}{2\eta}}t)\varphi''_{{X_1}}(\lambda^{-\frac{1}{2\eta}}t)t\right\}=\\
\nonumber&&\left(1/\eta - 1\right)\varphi^{1/\eta-2}_{{X_1}}(0)\left[\varphi'_{{X_1}}(0)\right]^2t + \varphi^{1/\eta-1}_{{X_1}}(0)\varphi''_{{X_1}}(0)t=0+i^2tE(X_1^2)=-t,
\end{eqnarray}
where $\varphi''_{X_1}(x)=d^2\varphi_{X_1}(x)/dx^2$. Hence, 
\begin{eqnarray}
\nonumber \lim\limits_{\lambda \to \infty} \varphi_{\widetilde{S}_{\lambda}}(t)= \exp\left\{-\frac{t^2}{2}\right\},\quad \forall s\in\mathbb R,
\end{eqnarray}
and the proof is completed by applying the  L\'evy's Continuity Theorem.
\end{proof}

Proposition \ref{padrao_COMPoisson} gives us an example in which a generalized Poisson summation yields a normal limiting distribution when properly normalized, in contrast with the mixed Poisson sums, where the limit is a normal mean-variance mixture.

\section*{Acknowledgements}
\noindent G. Oliveira thanks the partial financial support from {\it Coordena\c c\~ao de Aperfei\c coamento de Pessoal de N\'ivel Superior} (CAPES-Brazil). W. Barreto-Souza acknowledges support for his research from the KAUST Research Fund and NIH 1R01EB028753-01.

\section*{Appendix}

The elements of the gradient function $\nabla g$ associated to the function in (\ref{funcao g}) are given by
$$\nabla g(x, y, z) = 
\begin{pmatrix}
\dfrac{\partial g_1(x, y, z)}{\partial x}&\dfrac{\partial g_1(x, y, z)}{\partial y}&\dfrac{\partial g_1(x, y, z)}{\partial z}\\
\dfrac{\partial g_2(x, y, z)}{\partial x}&\dfrac{\partial g_2(x, y, z)}{\partial y}&\dfrac{\partial g_2(x, y, z)}{\partial z}\\
\dfrac{\partial g_3(x, y, z)}{\partial x}&\dfrac{\partial g_3(x, y, z)}{\partial y}&\dfrac{\partial g_3(x, y, z)}{\partial z}
\end{pmatrix} = 
\begin{pmatrix}
1&0&0\\
\nabla g_{21}&\nabla g_{22}&\nabla g_{23}\\
\nabla g_{31}&\nabla g_{32}&\nabla g_{33}
\end{pmatrix},
$$
where
\begin{eqnarray*}
&&\nabla g_{21} = -2x\Gamma\Big(h^{-1}(\omega) + 1\Big) + (y - x^2)\Gamma'\Big(h^{-1}(\omega) + 1\Big)\dfrac{d}{dx}\left\{h^{-1}(\omega)\right\},\\
&&\nabla g_{22} = \Gamma\Big(h^{-1}(\omega) + 1\Big) + (y - x^2)\Gamma'\Big(h^{-1}(\omega) + 1\Big)\dfrac{d}{dy}\left\{h^{-1}(\omega)\right\},\\
&&\nabla g_{23} = (y - x^2)\Gamma'\Big(h^{-1}(\omega) + 1\Big)\dfrac{d}{dz}\left\{h^{-1}(\omega)\right\},\quad
\nabla g_{31} = \dfrac{d}{dx}\left\{h^{-1}(\omega)\right\}, \\
&&\nabla g_{32} = \dfrac{d}{dy}\left\{h^{-1}(\omega)\right\},\quad
\nabla g_{33} = \dfrac{d}{dz}\left\{h^{-1}(\omega)\right\},
\end{eqnarray*}
for $\omega = \dfrac{z - 6x^2y + 5x^4}{6(y - x^2)^2}$, $(x,y,z)\in\mathbb R^3$, and
\begin{eqnarray*}
&&\Gamma'\Big(h^{-1}(\omega) + 1\Big) = \Gamma\Big(h^{-1}(\omega) + 1\Big)\Psi\Big(h^{-1}(\omega) + 1\Big),\\
&&\dfrac{d}{dx}\left\{h^{-1}(\omega)\right\} = \dfrac{1}{h'\Big(h^{-1}(\omega)\Big)}\left(\dfrac{4x^3y - 6xy^2 + 2xz)}{3(y - x^2)^3}\right),\\
&&\dfrac{d}{dy}\left\{h^{-1}(\omega)\right\} = \dfrac{1}{h'\Big(h^{-1}(\omega)\Big)}\left(\dfrac{-2x^4 + 3x^2y - z}{3(y - x^2)^3}\right),\\ 
&&\dfrac{d}{dz}\left\{h^{-1}(\omega)\right\} = \dfrac{1}{h'\Big(h^{-1}(\omega)\Big)6(y - x^2)^2},\\
&&h'\Big(h^{-1}(\omega)\Big) = \dfrac{2\Gamma^2\Big(h^{-1}(\omega) + 1\Big)\Big[\Psi\Big(h^{-1}(\omega) + 1\Big) - \Psi\Big(2h^{-1}(\omega) + 1\Big)\Big]}{\Gamma\Big(2h^{-1}(\omega) + 1\Big)},
\end{eqnarray*}
with $\Psi(\cdot)$ being the digamma function.

\bibliographystyle{spbasic}

\end{document}